\documentclass[12pt]{amsart}

\usepackage{amsmath,graphicx}
\usepackage{amscd}
\usepackage{eucal}
\usepackage{amssymb}
\usepackage{epic}
\usepackage{graphics}
\usepackage{color}
\usepackage{mathrsfs}
\usepackage{multirow}
\usepackage{subfigure}
\usepackage{graphicx}
\usepackage{mathtools}
\usepackage{hyperref}
\usepackage{enumitem}

\usepackage{changepage}

\numberwithin{equation}{section}
\numberwithin{table}{section}

\theoremstyle{plain}
\newtheorem{theorem}{Theorem}[section]
\newtheorem*{thm:A}{Theorem \ref{thm:A}}
\newtheorem*{thm:B}{Theorem \ref{thm:B}}
\newtheorem{lemma}{Lemma}[section]

\newtheorem{corollary}{Corollary}[section]

\theoremstyle{remark}
\newtheorem*{remark}{Remark}

\newtheorem{example}{Example}[section]

\newcommand{\ZZ}{{\mathbb{Z}}}
\newcommand{\QQ}{{\mathbb{Q}}}

\newcommand{\allone}{{\mathbf{1}}}

\newcommand{\exdiv}{\parallel}

\newcommand{\restr}[1]{|_{#1}}

\DeclareMathOperator{\im}{Im}

\DeclareMathOperator{\diag}{diag}
\DeclareMathOperator{\coker}{coker}

\title[Critical groups of SRGs]{Critical group structure from the parameters of a strongly regular graph.}
\author[Ducey et al.]{Joshua E. Ducey, David L. Duncan, Wesley J. Engelbrecht, Jawahar V. Madan, Eric Piato, Christina S. Shatford, Angela Vichitbandha}

\address{Department of Mathematics and Statistics, James Madison University, Harrisonburg, VA 22807, USA}
\email{duceyje@jmu.edu}
\email{duncandl@jmu.edu}
\email{engelber@dukes.jmu.edu}
\address{Department of Mathematics, Harvey Mudd College, Claremont, CA 91711, USA}
\email{jmadan@g.hmc.edu}
\address{Department of Mathematics, SUNY College at Geneseo, Geneseo, NY 14454, USA}
\email{esp6@geneseo.edu}
\address{Mathematics and Statistics Department, Connecticut College, New London, CT 06320, USA}
\email{cshatfor@conncoll.edu}
\address{Department of Mathematics, University of Kentucky, Lexington, KY 40506, USA}
\email{aavichitbandha@uky.edu}


\begin{document}
\begin{abstract}
We give simple arithmetic conditions that force the Sylow $p$-subgroup of the critical group of a strongly regular graph to take a specific form. These conditions depend only on the parameters $(v, k, \lambda, \mu)$ of the strongly regular graph under consideration. We give many examples, including how the theory can be used to compute the critical group of Conway's $99$-graph and to give an elementary argument that no $srg(28,9,0,4)$ exists.

\end{abstract}
\maketitle
\section{Introduction}

Given a finite, connected graph $\Gamma$, one can construct an interesting graph invariant $K(\Gamma)$ called the \textit{critical group}. This is a finite abelian group that captures non-trivial graph-theoretic information of $\Gamma$, such as the number of spanning trees of $\Gamma$; precise definitions are given in Section \ref{sec:prelims}. This group $K(\Gamma)$ goes by several other names in the literature (e.g., the \emph{Jacobian group} and the \emph{sandpile group}), reflecting its appearance in several different areas of mathematics and physics; see \cite{dino:finite} for a good introduction and \cite{glass} for a recent survey. Correspondingly, the critical group can be presented and studied by various methods. These methods include analysis of chip-firing games on the vertices of $\Gamma$ \cite{rotor}, framing the critical group in terms of the free group on the directed edges of $\Gamma$ subject to some natural relations \cite{deryagina}, computing (e.g., via unimodular row/column operators) the Smith normal form of a Laplacian matrix of the graph, and considering the underlying matroid of $\Gamma$ \cite{vince}.

Despite the variety of tools available, computing the critical group of an arbitrarily chosen graph can be computationally expensive. Instead, one often searches for families of graphs for which specific graph-theoretic knowledge can be used to streamline the computations involved. From this perspective, the \emph{strongly regular graphs} (\emph{srg}s) are a particularly interesting family. To paraphrase Peter Cameron, srgs lie on the boundary of the highly structured yet seemingly random. Computations have born witness to this, in that the critical groups of many subfamilies of srgs have been computed, while many more remain unknown. Examples of interesting subfamilies of srgs that have proven to be amenable to critical group computation include the Paley graphs \cite{sin:paley}, the $n \times n$ rook graphs \cite{rook}, Grassmann graphs on lines in projective space \cite{grassmann}, and Kneser graphs on $2$-element subsets \cite{kneser} (and the complements of all these). Some very recent progress deals with polar graphs \cite{raghu:polar} and the van Lint-Schrijver cyclotomic srgs \cite{raghu:cyclotomic}.


To each srg, one can associate parameters $(v, k, \lambda, \mu)$ describing the number and valence of the vertices, as well as adjacency information. The families of srgs listed above are each such that these parameters vary over the family. An alternative approach for studying \emph{srg}s is to fix the parameters $(v, k, \lambda, \mu)$ and explore what can be deduced about an srg with these parameters. It is this technique that is taken here; see also \cite[Section 3]{dino:snf} and \cite[Section 10]{biggs} for similar approaches. More specifically, we show that the parameter set $(v, k, \lambda, \mu)$ determines arithmetic conditions that constrain the Sylow $p$-subgroup of $K(\Gamma)$ for any strongly regular graph $\Gamma$ having these parameters. 

The aforementioned Sylow $p$-subgroup constraints arise through an extension of the analysis in \cite{brouwer:prank} of the $p$-ranks of the Laplacian matrix $L$. The need for such an extension stems from the observation that, though knowing the critical group of $\Gamma$ gives you the $p$-rank of $L$ for any prime $p$, the converse need not hold. That is, the $p$-rank of $L$ may not uniquely determine the Sylow $p$-subgroup of $K(\Gamma)$.  The smallest counterexample is the $4 \times 4$ rook graph and the Shrikhande graph.  These are both strongly regular graphs with parameters $(16,6,2,2)$ and both of their Laplacian matrices have $2$-rank equal to $6$.  However the critical group of the rook graph is
\[
\left ( \ZZ/8\ZZ \right )^{5} \oplus \left ( \ZZ/32\ZZ \right )^{4}
\]
while the Shrikhande graph has critical group
\[
\ZZ/2\ZZ \oplus \left ( \ZZ/8\ZZ \right )^{2} \oplus \left ( \ZZ/16\ZZ \right )^{2} \oplus \left ( \ZZ/32\ZZ \right )^{4}.
\]
Nevertheless, the critical groups of these graphs can be distinguished by considering their Sylow 2-subgroups (it happens to be the case in these examples that the Sylow 2-subgroup equals the full critical group).


The approach we take here may be of limited use in distinguishing non-isomorphic srgs with the same parameter set.  However, as we demonstrate in Example \ref{ex:28}, our approach can be applied to show that there cannot exist srgs with certain parameter sets.

\section{Preliminaries} \label{sec:prelims}
\subsection{Strongly Regular Graphs}\label{sec:StronglyRegularGraphs}
Let $\Gamma = \Gamma(\mathcal{V},\mathcal{E})$ denote a connected, finite undirected graph, as in the introduction.  If every vertex in $\mathcal{V}$ is adjacent to $k$ other vertices, we say that $\Gamma$ is \emph{$k$-regular}.  Fix an ordering of the vertices.  Then the \textit{adjacency matrix} $A=(a_{i,j})$ of $\Gamma$ is defined
\[
a_{i,j} = \begin{cases}  1, & \mbox{ if vertex $i$ and vertex $j$ are adjacent}\\
0, & \mbox{ otherwise.}
\end{cases}
\]
Let $D$ denote the $|\mathcal{V}| \times |\mathcal{V}|$ diagonal matrix with $(i,i)$-entry equal to the degree of vertex $i$.  The \emph{Laplacian matrix} of $\Gamma$ is $L = D - A$. We use $I$ and $J$ to denote, respectively, the identity matrix and the all-ones matrix of the appropriate size. Note that when $\Gamma$ is $k$-regular, we have $L = kI - A$.

A graph $\Gamma$ is \emph{strongly regular with parameters} $(v, k, \lambda, \mu)$ if:
\begin{itemize}[leftmargin=1cm]
\item $\Gamma$ has $v$ vertices,
\item $\Gamma$ is $k$-regular,
\item any two adjacent vertices have exactly $\lambda$ common neighbors, and,
\item any two non-adjacent vertices have exactly $\mu$ common neighbors.
\end{itemize}
We sometimes abbreviate this by writing that $\Gamma$ is an $srg(v,k,\lambda,\mu)$. 

We now recall several formulas and standard facts about the Laplacian $L$ of an $srg(v, k, \lambda, \mu)$; for more details, see, e.g., \cite[Chapter 9]{brouwer:spectra}. The all-ones vector spans the kernel of $L$, so $0$ is an eigenvalue of $L$ with multiplicity one. Aside from this $0$ eigenvalue, $L$ has exactly two other distinct eigenvalues that we denote by $r$ and $s$. These can be computed directly from the parameters $(v, k , \lambda, \mu)$, and can be shown to satisfy the quadratic matrix equation
\begin{equation}\label{eqn:srgB}
(L - rI)(L-sI) = \mu J.
\end{equation}
Example \ref{ex:28} in the next section shows how Equation \ref{eqn:srgB} can be a powerful tool for probing a particular graph. We will write $f$ and $g$ for the multiplicities of $r$ and $s$, respectively.

Of great interest is the existence question for strongly regular graphs.  The Handbook of Combinatorial Designs \cite[Chapter 11]{handbook} has a large list of feasible parameter sets, along with adjacency spectra and known graph constructions.  An up to date version of this list, with more information, is available at Andries Brouwer's website \cite{brouwer:website}. On Brouwer's list, the graph parameters are color coded green for those for which examples exist, red for those for which it is known that no graph exists, and yellow if the question is not yet decided.   Excluded from these lists are the ``boring'' strongly regular graphs, which are the disjoint unions of complete graphs or the complements of these. The disjoint unions of multiple complete graphs are excluded for us as well, by our connectedness assumption.

\subsection{Critical Groups}  The Laplacian $L$ can be viewed as defining a homomorphism of free abelian groups $L \colon \ZZ^{\mathcal{V}} \to \ZZ^{\mathcal{V}}$. Since $L$ has a kernel of rank one, it follows that the cokernel has (free) rank one as well. In particular, we have a decomposition of the form
\[
\ZZ^{\mathcal{V}} / \im(L) \cong K(\Gamma) \oplus \ZZ,
\]
with $K(\Gamma)$ a finite abelian group called the \textit{critical group} of $\Gamma$.  (If $\Gamma$ were not connected there would be more copies of $\ZZ$.)  The order of $K(\Gamma)$ is the number of spanning trees of the graph.  Isomorphic graphs have isomorphic critical groups, so the critical group is a graph invariant.  

From the matrix-tree theorem \cite[Prop. 1.3.4]{brouwer:spectra}, we have that the order of the critical group is the product of the nonzero Laplacian eigenvalues, divided by the number of vertices.  In the case that $\Gamma$ is an $srg(v,k,\lambda,\mu)$, this becomes the identity
\[
|K(\Gamma)| = \frac{r^{f} \cdot s^{g}}{v}.
\]
Moreover, one can use Equation \ref{eqn:srgB} to show that the product $rs$ kills $K(\Gamma)$. (It is a remarkable fact, proved by Lorenzini \cite[Prop. 2.6]{dino:snf}, that the product of the distinct nonzero Laplacian eigenvalues kills the critical group of \textit{any} graph.)

Let $p$ be a prime and write $K_{p}(\Gamma)$ for the Sylow $p$-subgroup of $K(\Gamma)$. By the structure theorem for finitely generated abelian groups, to determine $K(\Gamma)$, it suffices to determine $K_p(\Gamma)$ for each $p$ dividing the order of $K(\Gamma)$. A popular approach for identifying $K_{p}(\Gamma)$ is to make use of the {Smith normal form} of $L$, which we review now: There is a unique integer diagonal matrix $S=\mathrm{diag}(s_1, \ldots, s_v)$ with (i) nonnegative diagonal entries $s_i$ satisfying $s_i \vert s_{i+1}$ for $1 \le i < v$, and (ii) so that there exist unimodular matrices $U, V$ satisfying
\begin{equation}\label{eqn:unimodular}
ULV = S.
\end{equation}
Then $S$ is the \emph{Smith normal form} of $L$ and the $s_i$ are the \emph{invariant factors}. The name is appropriate since the cokernel of $L$ has invariant factor decomposition
\begin{equation}\label{eqn:coker}
\coker(L) \cong \ZZ/s_{1}\ZZ \oplus \cdots \oplus \ZZ/s_{v}\ZZ .
\end{equation}
It follows from our connectedness assumption that $s_v = 0$, while $s_i \neq 0$ for all $1 \le i < v$; in particular, the critical group can be read off from (\ref{eqn:coker}) by taking the first $v-1$ terms.

\begin{example}
Let $\Gamma$ denote the Petersen graph.  There is an ordering of the vertices so that the Laplacian matrix for $\Gamma$ is
\[
L = \begin{bmatrix}
3& -1&  &  & -1& -1&  &  &  &  \\
-1&  3& -1&  &  &  & -1&  &  &  \\
 & -1&  3& -1&  &  &  & -1&  &  \\
 &  & -1&  3& -1&  &  &  & -1&  \\
-1&  &  & -1&  3&  &  &  &  & -1\\
-1&  &  &  &  &  3&  & -1& -1&  \\
 & -1&  &  &  &  &  3&  & -1& -1\\
 &  & -1&  &  & -1&  &  3&  & -1\\
 &  &  & -1&  & -1& -1&  &  3&  \\
 &  &  &  & -1&  & -1& -1&  &  3
 \end{bmatrix}.
\]
This matrix has Smith normal form
\[
 \diag(1,1,1,1,1,2,10,10,10,0)
 \]
 from which it follows that $K(\Gamma) \cong \ZZ/2\ZZ \oplus \left ( \ZZ/10\ZZ \right )^{3}$.  Equivalently, the critical group can be written relative to its elementary divisor decomposition as $K(\Gamma) \cong \left ( \ZZ/2\ZZ \right )^{4} \oplus \left ( \ZZ/5\ZZ \right )^{3}$, which is easily read off by looking at the invariant factors. The two summands appearing in this latter description are the Sylow 2- and 5-subgroups of $K(\Gamma)$, respectively. 
 
\end{example}

We will repeatedly use the following notation:  For a fixed graph $\Gamma$ and prime $p$, we define $e_{i}$ to be the number of invariant factors of $L$ that are divisible by $p^{i}$ but not divisible by $p^{i+1}$. Notice that $e_{0}$ is the $p$-rank of $L$ (the rank when viewed as a matrix over the field of $p$ elements). For $i>0$, the integer $e_{i}$ is the multiplicity of $\ZZ/p^{i}\ZZ$ in the elementary divisor decomposition of the critical group. We will refer to the $e_i$ as the \emph{($p$-elementary divisor) multiplicities}, and note that they uniquely determine $K_p(\Gamma)$.

To compute these multiplicities we can use the following construction.  For fixed $p$ and $i \geq 0$, define
\[
M_{i} = \left \{ x \in \ZZ^{\mathcal{V}} \, \vert \, Lx \mbox{ is divisible by } p^{i} \right \}
\]
and
\[
N_{i} = \left \{ p^{-i} Lx  \, \vert \, x \in M_{i} \right \}.
\]
We use bar notation to denote entry-wise reduction modulo $p$ of vectors and matrices.  By considering the $\ZZ$-bases of $\ZZ^{\mathcal{V}}$ defined by the unimodular matrices $U,V$ in Equation \ref{eqn:unimodular} one sees that
\begin{eqnarray}
\dim_{p} \overline{M_{i}} = 1 + \sum_{j \geq i} e_{j} \label{eqn:Mi}\\
\dim_{p} \overline{N_{i}} = \sum_{0 \leq j \leq i} e_{j}.\label{eqn:Ni}
\end{eqnarray}
For a reference, see \cite[Prop. 13.8.2, 13.8.3]{brouwer:spectra}.

Our main tool is the following lemma, which relates the spectrum of $L$ to the critical group $K(\Gamma)$. Recall that $p^{i} \exdiv n$ means that $p^{i} \mid n$ and $p^{i+1} \nmid n$.

\begin{lemma}\label{lem:eigen}
Let $\Gamma$ be a connected graph, fix a prime $p$, and let $e_{i}$ be the multiplicity of $p^{i}$ as an elementary divisor of the Laplacian $L$. Let $\eta$ be an eigenvalue of $L$ with multiplicity $m$, and assume that $\eta$ is an integer. 

\begin{itemize}
\item If $p^{i} \mid \eta$, then $m \leq 1 + \sum_{j \geq i} e_{j}$.
\item If $p^{i} \exdiv \eta $, then $m \leq \sum_{0 \leq j \leq i} e_{j}$.
\end{itemize}
\end{lemma}

\begin{proof}
Let $V_{\eta}$ denote the $\eta$-eigenspace of $L$, when viewed as a matrix over the rational numbers $\QQ$.  The intersection $V_{\eta} \cap \ZZ^{\mathcal{V}}$ is a pure $\ZZ$-submodule of $\ZZ^{\mathcal{V}}$ of rank $m$, and so $\dim_{p} \overline{V_{\eta} \cap \ZZ^{\mathcal{V}}} = m$.  Since $p^i$ divides $\eta$, we have $V_{\eta} \cap \ZZ^{\mathcal{V}} \subseteq M_{i}$ and hence $\overline{V_{\eta} \cap \ZZ^{\mathcal{V}}} \subseteq \overline{M_{i}}$. It follows that
\[
m = \dim_{p} \overline{V_{\eta} \cap \ZZ^{\mathcal{V}}} \leq \dim_{p} \overline{M_{i}} = 1 + \sum_{j \geq i} e_{j}.
\]
For the second claim, write $\eta = x p^i$ for some integer $x$. Then $V_{\eta} \cap \ZZ^{\mathcal{V}} \subseteq M_{i}$ implies $x(V_{\eta} \cap \ZZ^{\mathcal{V}}) \subseteq N_{i}$, and so $\overline{x(V_{\eta} \cap \ZZ^{\mathcal{V}})} \subseteq \overline{N_{i}}$. The assumption that $p^{i} \exdiv \eta$ implies that $x$ is invertible mod $p$. Thus 
$$m = \dim_{p} \overline{V_{\eta} \cap \ZZ^{\mathcal{V}}} = \dim_p \overline{x(V_{\eta} \cap \ZZ^{\mathcal{V}})}     \leq \dim_p \overline{N_i}  = \sum_{0 \leq j \leq i} e_{j}.$$
\end{proof}


\section{Sylow $p$-subgroup structure} \label{sec:elem}
Throughout this section, $\Gamma$ denotes a connected $srg(v, k, \lambda, \mu)$ with Laplacian matrix $L$. As we have discussed, $L$ has two non-zero eigenvalues $r$ and $s$, and we denote by $f$ and $g$ their respective multiplicities. We assume that $r$ and $s$ are \emph{integers}, which is the case for any srg unless it is a conference graph \cite[Theorem 9.1.3]{brouwer:spectra}. We fix a prime $p$ dividing $|K(\Gamma)|$ and we write $K_{p}(\Gamma)$ for the Sylow $p$-subgroup of $K(\Gamma)$. Recall that $e_{i}$ denotes the multiplicity of $p^{i}$ as an elementary divisor of $L$; in particular, $e_{0}$ is the $p$-rank of $L$.

\begin{theorem}\label{thm:1}
Suppose $p \nmid r$, and let $a, \gamma$ be the (unique) nonnegative integers so that $p^{a} \exdiv s$ and $p^{\gamma} \exdiv v$.  Then
\[
K_{p}(\Gamma) \cong \ZZ/p^{a-\gamma}\ZZ \oplus \left ( \ZZ/p^{a}\ZZ \right )^{g-1}.
\]
The same statement holds if the roles of $r$ and $s$ are interchanged, and the roles of $f$ and $g$ are interchanged.
\end{theorem}
\begin{proof}
We have assumed that $p$ divides $\vert K(\Gamma) \vert = r^f  s^g / v$, so the hypotheses imply that $a \geq 1$. Similarly, since $rs$ kills the critical group and $p^{a} \exdiv rs$ we have
\[
K_p(\Gamma) \cong \left ( \ZZ/p\ZZ \right )^{e_{1}} \oplus \left ( \ZZ/p^{2}\ZZ \right )^{e_{2}} \oplus \cdots \oplus \left ( \ZZ/p^{a}\ZZ \right )^{e_{a}}. 
\] 
From the Smith normal form of $L$, we see that $e_{0} + e_{1} + \cdots + e_{a} + 1 = v$ is the number of diagonal entries in the Smith normal form. Similarly, by diagonalizing $L$, we see $f + g + 1 = v$. This gives
\begin{equation}\label{eqn:kills}
e_{0} + e_{1} + \cdots + e_{a}  = f + g.
\end{equation}
The order of $K_{p}(\Gamma)$ we get from the matrix-tree theorem:
\[
|K_{p}(\Gamma)| = \frac{\left ( p^{a} \right )^{g}}{p^{\gamma}}.
\]
This order can be alternatively expressed in terms of the elementary divisor multiplicities, from which we obtain
\begin{equation}\label{eqn:kirchhoff}
e_{1} + 2e_{2} + \cdots + ae_{a} = ag - \gamma.
\end{equation}
Applying Lemma \ref{lem:eigen} to the $s$-eigenspace of $L$ we have
\[
g \leq e_{a} + 1.
\]
In fact, we always have
\begin{equation}\label{eqn:ea}
g-1 \leq e_{a} \leq g.
\end{equation}
For suppose that $e_{a} > g$.  Then $ae_{a} > ag \geq ag - \gamma \geq ae_{a}$,
where the last inequality follows from Equation \ref{eqn:kirchhoff}.  This is impossible therefore the bound \ref{eqn:ea} holds.  

Next we will see that the $p$-rank $e_{0}$ must equal $f$ or $f+1$. In the case that $e_{a} = g$, Equation \ref{eqn:kills} gives
\[e_{0} + \cdots + e_{a-1} = f
\]
 and so $e_{0} \leq f$.  By Lemma \ref{lem:eigen} applied to the $r$-eigenspace, we have $f \leq \dim \overline{N_{0}} = e_{0}$.  Thus $e_{0} = f$ and we see $e_{i} = 0$ for $i \neq 0, a$ by Equation \ref{eqn:kills}.  So in this case
\[
K_{p}(\Gamma) \cong \left ( \ZZ/p^{a}\ZZ \right )^{g},
\]
which agrees with the statement of the theorem since Equation \ref{eqn:kirchhoff} forces $\gamma$ to be zero.

Now consider the case $e_{a} = g-1$.  From Equation \ref{eqn:kills} we get 
\[
e_{0} + \cdots + e_{a-1} = f+1
\]
and so $e_{0} \leq f+1$.  As before we also have $f \leq e_{0}$.  It turns out that both $e_{0}=f$ and $e_{0}=f+1$ are possible (more about this in the next corollary).  In the case that $e_{0}=f+1$, we are forced to have $e_{i}=0$ for $i \neq 0,a$ and we get 
\[
K_{p}(\Gamma) \cong \left ( \ZZ/p^{a}\ZZ \right )^{g-1}
\]
which agrees with the statement of the theorem, since now Equation \ref{eqn:kirchhoff} forces $\gamma = a$.

Finally, if $e_{a} = g-1$ and $e_{0} = f$, we see that Equation \ref{eqn:kills} becomes
\[
e_{1} + \cdots + e_{a-1} = 1.
\]
This means that there is some $i \neq 0, a$ with $e_i = 1$ and $e_j = 0$ for $j \neq 0, i, a$. We can identify the distinguished subscript $i$ by looking carefully at Equation \ref{eqn:kirchhoff}:
\begin{align*}
e_{1} + 2e_{2} + \cdots + (a-1)e_{a-1} &= ag - \gamma - ae_{a}\\
&= ag - \gamma - a(g-1)\\
&= a - \gamma.
\end{align*}
Thus we see that $i = a - \gamma$.  We have shown in this case that
\[
K_{p}(\Gamma) \cong \ZZ/p^{a-\gamma}\ZZ \oplus \left ( \ZZ/p^{a}\ZZ \right )^{g-1},
\]
as desired.
\end{proof}

The statement of Theorem \ref{thm:1} is simple, but as the proof shows, the distinguished summand $\ZZ/p^{a-\gamma}\ZZ$ can be absorbed into the others (when $\gamma = 0$) or can disappear entirely (when $\gamma = a$).  We also saw that $\gamma$ is forced by the values of $e_{a}$ and $e_{0}$.  In \cite[Section 3]{brouwer:prank}, the authors calculate the $p$-ranks of matrices in a class that includes our $L$ (under the hypotheses of Theorem \ref{thm:1}) and they show that $e_{0}$ is determined by whether or not $p$ divides $\mu$.  We record this information in case it is of organizational value to the reader.

\begin{corollary}\label{cor:1}
Suppose $p \nmid r$ and let $a, \gamma$ be the (unique) nonnegative integers so that $p^{a} \exdiv s$ and $p^{\gamma} \exdiv v$.  Then exactly one of the following hold:
\begin{enumerate}
\item $\gamma = 0$, $p \mid \mu$, $e_{0} = f$ and $K(\Gamma) \cong \left ( \ZZ/p^{a}\ZZ \right )^{g}$,
\item $0 < \gamma < a$, $p \mid \mu$, $e_{0} = f$ and $K(\Gamma) \cong \ZZ/p^{a-\gamma}\ZZ \oplus \left ( \ZZ/p^{a}\ZZ \right )^{g-1}$,
\item $\gamma = a$, $p \nmid \mu$, $e_{0} = f+1$ and $K(\Gamma) \cong \left ( \ZZ/p^{a}\ZZ \right )^{g-1}$.
\end{enumerate}
The same statement holds if the roles of $r$ and $s$ are interchanged, and the roles of $f$ and $g$ are interchanged.
\end{corollary}

%
%

Let's apply these theorems with a few examples.
\begin{example}
It is unknown whether there exists a strongly regular graph $\Gamma$ with parameters $(190,84,33,40)$.  If such a graph exists then its nonzero Laplacian eigenvalues and multiplicities would have to be $r^{f} = 80^{133}$ and $s^{g} = 95^{56}$ (we are writing the multiplicities as exponents, as is custom in much of the literature).  Since $r = 16 \cdot 5$ and $s = 5 \cdot 19$, we can use the theorem above to compute the Sylow $2$- and $19$-subgroups of $K(\Gamma)$ (though it is easy to see that $K_{19}(\Gamma)$ is elementary abelian).  Let's compute $K_{2}(\Gamma)$:
\[
K_{2}(\Gamma) \cong \ZZ/2^{4-1}\ZZ \oplus \left ( \ZZ/2^{4}\ZZ \right )^{133-1} = \ZZ/8\ZZ \oplus \left ( \ZZ/16\ZZ \right )^{132}.
\]
\end{example}
\begin{example}
Conway's $99$-graph problem asks whether there exists a strongly regular graph $\Gamma$ with parameters $(99,14,1,2)$.  The nonzero Laplacian eigenvalues and multiplicities of such a graph would have to be $r^{f} = 11^{54}$ and $s^{g} = 18^{44}$.  Since $r$ and $s$ are relatively prime, we can apply our theorems to obtain the complete critical group.  We find
\[
K(\Gamma) \cong \left ( \ZZ/11\ZZ \right )^{53} \oplus \left ( \ZZ/2\ZZ \right )^{44} \oplus \left ( \ZZ/9\ZZ \right )^{43}.
\]
\end{example}

When $p$ divides both $r$ and $s$, it can occur that the critical group depends on the structure of the graph. Our next theorem shows that, in the simplest such case, this dependence is encoded entirely in the value of $e_0$.
  
\begin{theorem}\label{thm:3}
Suppose $p \exdiv r$ and $p \exdiv s$, and let $\gamma$ be the (unique) nonnegative integer so that $p^{\gamma} \exdiv v$.  Then
\[
K_{p}(\Gamma) \cong \left ( \ZZ/p\ZZ \right )^{f + g + \gamma - 2e_{0}} \oplus \left ( \ZZ/p^{2}\ZZ \right )^{e_{0} - \gamma}.
\]
\end{theorem}
\begin{proof}
The matrix-tree theorem gives us $\vert K_{p}(\Gamma) \vert = p^{f+g-\gamma}$, and since $p^{2} \exdiv rs$ we have
\[
K_{p}(\Gamma) \cong \left ( \ZZ/p\ZZ \right )^{e_{1}} \oplus \left ( \ZZ/p^{2}\ZZ \right )^{e_{2}}.
\]
In terms of the elementary divisor multiplicities, this can be expressed as
\begin{align*}
e_{0} + e_{1} + e_{2} &= f + g \\
e_{1} + 2e_{2} &= f + g - \gamma.
\end{align*}
Thus knowing any one of $e_{0}, e_{1}, e_{2}$ determines the others.  Taking $e_{0}$ to be free we compute 
\begin{align*}
e_{1} &= f + g + \gamma - 2e_{0} \\
e_{2} &= e_{0} - \gamma
\end{align*}
and the theorem is proved.
\end{proof}

\begin{example}
Consider the parameter set $(25,12,5,6)$.  We have that $r^{f} = 10^{12}$ and $s^{g} = 15^{12}$, so the prime $p=5$ is of particular interest.  This is, in fact, the first parameter set for which the hypotheses of Theorem \ref{thm:3} are satisfied and for which there is more than one graph with these parameters.  There are exactly $15$ strongly regular graphs with these parameters and adjacency matrices for them can be found at Ted Spence's website \cite{spence:website}. We let $\Gamma_{1}$ denote the graph having adjacency matrix given by the first matrix on Spence's list, which we reproduce here for convenience:

\[\mbox{\tiny$
\begin{bmatrix} 
0 & 1 & 1 & 1 & 1 & 1 & 1 & 1 & 1 & 1 & 1 & 1 & 1 & 0 & 0 & 0 & 0 & 0 & 0 & 0 & 0 & 0 & 0 & 0 & 0\\
1 & 0 & 1 & 1 & 1 & 1 & 1 & 0 & 0 & 0 & 0 & 0 & 0 & 1 & 1 & 1 & 1 & 1 & 1 & 0 & 0 & 0 & 0 & 0 & 0\\
1 & 1 & 0 & 1 & 1 & 1 & 1 & 0 & 0 & 0 & 0 & 0 & 0 & 0 & 0 & 0 & 0 & 0 & 0 & 1 & 1 & 1 & 1 & 1 & 1\\
1 & 1 & 1 & 0 & 0 & 0 & 0 & 1 & 1 & 1 & 0 & 0 & 0 & 1 & 1 & 1 & 0 & 0 & 0 & 1 & 1 & 1 & 0 & 0 & 0\\
1 & 1 & 1 & 0 & 0 & 0 & 0 & 1 & 0 & 0 & 1 & 1 & 0 & 1 & 0 & 0 & 1 & 1 & 0 & 1 & 0 & 0 & 1 & 1 & 0\\
1 & 1 & 1 & 0 & 0 & 0 & 0 & 0 & 1 & 0 & 1 & 0 & 1 & 0 & 1 & 0 & 1 & 0 & 1 & 0 & 1 & 0 & 1 & 0 & 1\\
1 & 1 & 1 & 0 & 0 & 0 & 0 & 0 & 0 & 1 & 0 & 1 & 1 & 0 & 0 & 1 & 0 & 1 & 1 & 0 & 0 & 1 & 0 & 1 & 1\\
1 & 0 & 0 & 1 & 1 & 0 & 0 & 0 & 1 & 1 & 1 & 0 & 0 & 1 & 0 & 1 & 0 & 0 & 1 & 0 & 0 & 0 & 1 & 1 & 1\\
1 & 0 & 0 & 1 & 0 & 1 & 0 & 1 & 0 & 1 & 0 & 1 & 0 & 0 & 1 & 0 & 1 & 1 & 0 & 0 & 1 & 0 & 0 & 1 & 1\\
1 & 0 & 0 & 1 & 0 & 0 & 1 & 1 & 1 & 0 & 0 & 0 & 1 & 0 & 0 & 0 & 1 & 1 & 1 & 1 & 0 & 1 & 1 & 0 & 0\\
1 & 0 & 0 & 0 & 1 & 1 & 0 & 1 & 0 & 0 & 0 & 1 & 1 & 1 & 1 & 0 & 0 & 0 & 1 & 1 & 0 & 1 & 0 & 0 & 1\\
1 & 0 & 0 & 0 & 1 & 0 & 1 & 0 & 1 & 0 & 1 & 0 & 1 & 0 & 1 & 1 & 0 & 1 & 0 & 1 & 1 & 0 & 0 & 1 & 0\\
1 & 0 & 0 & 0 & 0 & 1 & 1 & 0 & 0 & 1 & 1 & 1 & 0 & 1 & 0 & 1 & 1 & 0 & 0 & 0 & 1 & 1 & 1 & 0 & 0\\
0 & 1 & 0 & 1 & 1 & 0 & 0 & 1 & 0 & 0 & 1 & 0 & 1 & 0 & 0 & 1 & 1 & 1 & 0 & 0 & 1 & 1 & 0 & 0 & 1\\
0 & 1 & 0 & 1 & 0 & 1 & 0 & 0 & 1 & 0 & 1 & 1 & 0 & 0 & 0 & 1 & 1 & 0 & 1 & 1 & 0 & 1 & 0 & 1 & 0\\
0 & 1 & 0 & 1 & 0 & 0 & 1 & 1 & 0 & 0 & 0 & 1 & 1 & 1 & 1 & 0 & 0 & 0 & 1 & 0 & 1 & 0 & 1 & 1 & 0\\
0 & 1 & 0 & 0 & 1 & 1 & 0 & 0 & 1 & 1 & 0 & 0 & 1 & 1 & 1 & 0 & 0 & 1 & 0 & 0 & 0 & 1 & 1 & 1 & 0\\
0 & 1 & 0 & 0 & 1 & 0 & 1 & 0 & 1 & 1 & 0 & 1 & 0 & 1 & 0 & 0 & 1 & 0 & 1 & 1 & 1 & 0 & 0 & 0 & 1\\
0 & 1 & 0 & 0 & 0 & 1 & 1 & 1 & 0 & 1 & 1 & 0 & 0 & 0 & 1 & 1 & 0 & 1 & 0 & 1 & 0 & 0 & 1 & 0 & 1\\
0 & 0 & 1 & 1 & 1 & 0 & 0 & 0 & 0 & 1 & 1 & 1 & 0 & 0 & 1 & 0 & 0 & 1 & 1 & 0 & 1 & 1 & 1 & 0 & 0\\
0 & 0 & 1 & 1 & 0 & 1 & 0 & 0 & 1 & 0 & 0 & 1 & 1 & 1 & 0 & 1 & 0 & 1 & 0 & 1 & 0 & 0 & 1 & 0 & 1\\
0 & 0 & 1 & 1 & 0 & 0 & 1 & 0 & 0 & 1 & 1 & 0 & 1 & 1 & 1 & 0 & 1 & 0 & 0 & 1 & 0 & 0 & 0 & 1 & 1\\
0 & 0 & 1 & 0 & 1 & 1 & 0 & 1 & 0 & 1 & 0 & 0 & 1 & 0 & 0 & 1 & 1 & 0 & 1 & 1 & 1 & 0 & 0 & 1 & 0\\
0 & 0 & 1 & 0 & 1 & 0 & 1 & 1 & 1 & 0 & 0 & 1 & 0 & 0 & 1 & 1 & 1 & 0 & 0 & 0 & 0 & 1 & 1 & 0 & 1\\
0 & 0 & 1 & 0 & 0 & 1 & 1 & 1 & 1 & 0 & 1 & 0 & 0 & 1 & 0 & 0 & 0 & 1 & 1 & 0 & 1 & 1 & 0 & 1 & 0\end{bmatrix}$}.
\]

  For another $srg(25, 12, 5, 6)$, we let $\Gamma_{2}$ be the Paley graph on $25$ vertices.  Using SAGE, we compute:
\[
K_{5}(\Gamma_{1}) \cong \left ( \ZZ/5\ZZ \right )^{2} \oplus \left ( \ZZ/25\ZZ \right )^{10} \quad \mbox{(so $e_{0} = 12$)}
\]
and
\[
K_{5}(\Gamma_{2}) \cong \left ( \ZZ/5\ZZ \right )^{8} \oplus \left ( \ZZ/25\ZZ \right )^{7} \quad \mbox{(so $e_{0} = 9$)}.
\]
Our Theorem \ref{thm:3} predicts 
\[
K_{5}(\Gamma) \cong \left ( \ZZ/5\ZZ \right )^{26 - 2e_{0}} \oplus \left ( \ZZ/25\ZZ \right )^{e_{0}-2},
\]
which agrees with these computations.
\end{example}
\begin{example}
It is unknown whether there exists a strongly regular graph $\Gamma$ with parameters $(88,27,6,9)$. By the results above, the critical group would be specified uniquely by $3$-rank. Indeed, if such a graph existed, we would have $r^{f} = 24^{55}$ and $s^{g} = 33^{32}$. Theorem \ref{thm:1} specifies the Sylow 2- and 11-subgroups, so the only mystery in knowing $K(\Gamma)$ is knowing $K_{3}(\Gamma)$, which is given in terms of the 3-rank by Theorem \ref{thm:3}: 
\[
K_{3}(\Gamma) \cong \left ( \ZZ/3\ZZ \right )^{87 - 2e_{0}} \oplus \left ( \ZZ/9\ZZ \right )^{e_{0}}.
\]
\end{example}

To summarize thus far:  under the hypotheses of Theorem \ref{thm:1} the structure of $K_{p}(\Gamma)$ is forced, and under the hypothesis of Theorem \ref{thm:3} the $p$-rank of $L$ determines $K_{p}(\Gamma)$.  Under the hypotheses of the next theorem, the $p$-rank of $L$ determines $K_{p}(\Gamma)$ to within two possibilities.

\begin{theorem}\label{thm:4}
Suppose $p \exdiv r$ and $p^{2} \exdiv s$, and let $\gamma$ be the (unique) nonnegative integer so that $p^{\gamma} \exdiv v$.  Then either
\[
K_{p}(\Gamma) \cong \left ( \ZZ/p\ZZ \right )^{f-e_{0}} \oplus \left ( \ZZ/p^{2}\ZZ \right )^{g + \gamma - e_{0}} \oplus \left ( \ZZ/p^{3}\ZZ \right )^{e_{0} - \gamma}
\]
or
\[
K_{p}(\Gamma) \cong \left ( \ZZ/p\ZZ \right )^{f+1-e_{0}} \oplus \left ( \ZZ/p^{2}\ZZ \right )^{g + \gamma - 2 - e_{0}} \oplus \left ( \ZZ/p^{3}\ZZ \right )^{e_{0} - \gamma + 1}.
\]
Furthermore, if $\gamma = 0$ then 
\[
K_{p}(\Gamma) \cong \left ( \ZZ/p\ZZ \right )^{f-e_{0}} \oplus \left ( \ZZ/p^{2}\ZZ \right )^{g - e_{0}} \oplus \left ( \ZZ/p^{3}\ZZ \right )^{e_{0}}.
\]
The same statement holds if the roles of $r$ and $s$ are interchanged, and the roles of $f$ and $g$ are interchanged.
\end{theorem}

\begin{proof}
Since $p \exdiv r$ and $p^{2} \exdiv s$, we have $p^{3} \exdiv rs$ and so
\[
K_{p}(\Gamma) \cong \left ( \ZZ/p\ZZ \right )^{e_{1}} \oplus \left ( \ZZ/p^{2}\ZZ \right )^{e_{2}} \oplus \left ( \ZZ/p^{3}\ZZ \right )^{e_{3}}.
\]
From this general form and the matrix-tree theorem we get the equations
\begin{align}
e_{0} + e_{1} + e_{2} + e_{3}  &= f + g  \nonumber \\
e_{1} + 2e_{2} + 3e_{3} &= f + 2g - \gamma. \label{eqn:extra}
\end{align}
Applying Lemma \ref{lem:eigen}, we have the bounds
\begin{align*}
f &\leq \dim \overline{N_{1}} = e_{0} + e_{1} \\
g &\leq \dim \overline{M_{2}} = e_{2} + e_{3} + 1. 
\end{align*}
The left sides of the above inequalities sum to $f+g$, while the right sides sum to $f+g+1$.  Thus we have our two possibilities: 
\[
f = e_{0} + e_{1} \mbox{ and } g+1 = e_{2} + e_{3} + 1
\]
or
\[
f+1 = e_{0} + e_{1} \mbox{ and } g = e_{2} + e_{3} + 1.
\]
In either case, with these two equations and Equation \ref{eqn:extra} we see that knowing one of $e_{0}, e_{1}, e_{2}, e_{3}$ forces the values of the others.  The first part of the theorem follows.

Now assume that $\gamma = 0$.  We just want to show that $K_{p}(\Gamma)$ must be the first possibility in the statement of the theorem.  Let $\ZZ_{(p)}$ be the ring of $p$-local integers, i.e. rational numbers that can be written as fractions with denominators coprime to $p$.  We can view $L$ has having entries coming from $\ZZ_{(p)}$ and if we do this, then $L$ defines a homomorphism of free $\ZZ_{(p)}$-modules
\[
L \colon \ZZ_{(p)}^{\mathcal{V}} \to \ZZ_{(p)}^{\mathcal{V}}.
\]
The Smith normal form of $L$ over this ring is the same as over the integers, but as primes different from $p$ are now units we may ignore them.  One advantage of this point of view is the following.  Since the number of vertices is not divisible by $p$, we have the decomposition
\[
\ZZ_{(p)}^{\mathcal{V}} = \ZZ_{(p)}\allone \oplus Y,
\]
where $Y = \big\{ \sum_{v \in \mathcal{V} }a_{v}v \in \ZZ_{(p)}^{\mathcal{V}} \, \vert \, \sum_{v \in \mathcal{V}} a_{v} = 0 \big\}$.  The Laplacian map respects this decomposition and this means that the $p$-elementary divisor multiplicities are the same for both $L$ and the restricted map
\[
L\restr{Y} \colon Y \to Y.
\]
The transformation defined by the all-ones matrix $J$ is zero on $Y$, therefore we get from Equation \ref{eqn:srgB}
\[
L\restr{Y} \left ( L\restr{Y} - (r+s)I \right ) = -rsI.
\]
Since $p^{3} \exdiv rs$, the equation above shows a symmetry of Smith normal forms:  the multiplicity of $p^{i}$ as an elementary divisor of $L\restr{Y}$ is equal to the multiplicity of $p^{3-i}$ as an elementary divisor of $L\restr{Y} - (r+s)I$.  Since $L\restr{Y}$ and $L\restr{Y} - (r+s)I$ are congruent modulo $p$, they must have the same $p$-rank.  The last two sentences imply that $e_{0} = e_{3}$ for our Laplacian $L$, so $K_{p}(\Gamma)$ must take the first form in the statement of the theorem.
\end{proof}


\begin{example}  The famous missing Moore graph would have to be an $srg(3250, 57, 0, 1)$, if it exists.  From these parameters, we have $r^{f} = 50^{1729} $ and $s^{g} = 65^{1520}$, and the interesting prime is $p=5$. From Theorem \ref{thm:4}, we get
\[
K_{5}(\Gamma) \cong \left ( \ZZ/5\ZZ \right )^{1520-e_{0}} \oplus \left ( \ZZ/25\ZZ \right )^{1732 - e_{0}} \oplus \left ( \ZZ/125\ZZ \right )^{e_{0} - 3}
\]
or
\[
K_{5}(\Gamma) \cong \left ( \ZZ/5\ZZ \right )^{1521-e_{0}} \oplus \left ( \ZZ/25\ZZ \right )^{1730 - e_{0}} \oplus \left ( \ZZ/125\ZZ \right )^{e_{0} - 2}.
\] 
(Note $\gamma = 3$.) This example first appeared in \cite{ducey:moore}.
\end{example}

\begin{example}
The Schl\"{a}fli graph is the unique $srg(27,16,10,8)$; denote it by $\Gamma$.  We have $r^{f} = 12^{6}$ and $s^{g} = 18^{20}$.  We can apply Theorem \ref{thm:4} to the prime $p=2$, and since $\gamma=0$ we must have
\[
K_{2}(\Gamma) \cong \left ( \ZZ/2\ZZ \right )^{20 - e_{0}} \oplus \left ( \ZZ/4\ZZ \right )^{6 - e_{0}} \oplus \left ( \ZZ/8\ZZ \right )^{e_{0}}.
\]
Using SAGE we find that the $2$-rank of $L$ is $6$ and also that 
\[
K_{2}(\Gamma) \cong \left ( \ZZ/2\ZZ \right )^{14} \oplus \left ( \ZZ/8\ZZ \right )^{6}
\]
which matches our prediction.
\end{example}

\begin{example}\label{ex:chang}
Let $\Gamma_{1}$ denote the complement of any one of the three Chang graphs.  Let $\Gamma_{2}$ denote the Kneser graph on the $2$-subsets of an $8$-element set (so adjacent when disjoint).  Both of these graphs are examples of an $srg(28, 15, 6, 10)$.  We have $r^{f} = 14^{20}$ and $s^{g} = 20^{7}$, and so Theorem \ref{thm:4} applies to the prime $p=2$ (note $\gamma = 2$).  

According to SAGE, the Laplacian of $\Gamma_{1}$ has $2$-rank equal to $8$ and
\[
K_{2}(\Gamma_{1}) \cong \left ( \ZZ/2\ZZ \right )^{12} \oplus \ZZ/4\ZZ \oplus \left ( \ZZ/8\ZZ \right )^{6}.
\]
Similarly, for $\Gamma_{2}$, the computer tells us that the Laplacian $2$-rank is $7$ and
\[
K_{2}(\Gamma_{2}) \cong \left ( \ZZ/2\ZZ \right )^{14} \oplus \left ( \ZZ/8\ZZ \right )^{6}.
\]
This illustrates that both of the cases described in Theorem \ref{thm:4} can occur.
\end{example}

\begin{remark}
Checking many strongly regular graphs on up to $36$ vertices (we did not check all of the $32548$ graphs with parameters $(36,15,6,6)$) the authors have not found a pair of graphs with the same parameters, the same $p$-rank, and demonstrating the separate cases of Theorem \ref{thm:4} (note the $2$-ranks are different in Example \ref{ex:chang}). So maybe, even under the hypotheses of Theorem \ref{thm:4}, the $p$-rank does determine $K_{p}(\Gamma)$.
\end{remark}

Our final example applies the theory to give an elementary proof that no $srg(28,9,0,4)$ exists.

\begin{example}\label{ex:28}
Suppose that a strongly regular graph with parameters $(28,9,0,4)$ exists. Denote it by $\Gamma$, and let $L$ be its Laplacian, which we may view as a matrix by ordering the vertices. We must have $r^f=8^{21}$ and $s^g=14^{6}$.  The matrix equation \ref{eqn:srgB} reads
\begin{equation}
\label{eqn1}
(L-14I)(L - 8I) = 4J,
\end{equation}
where $J$ is the matrix of all-ones.

To motivate our choices below, we note that this graph is red on Brouwer's list.  We know it does not actually exist since it contradicts the `absolute bound' $28 \leq 6(6+3)/2$ (it also contradicts one of the Krein inequalities).  If we are looking for a Smith normal form or $p$-rank argument, this suggests that we might look at the prime 7, which divides the eigenvalue with multiplicity that is too small according to this bound.

Returning to our argument, let $F = \ZZ/7\ZZ$ be the field of $7$ elements, and write $\overline{L}$ for the matrix $L$ with entries viewed as coming from $F$.  From Corollary \ref{cor:1}, the rank of $\overline{L}$ is 22, and so the dimension of $\ker \overline{L}$ is 6.  We can thus arrive at a contradiction if we exhibit more than 6 independent vectors in $\ker \overline{L}$. 

Fix two adjacent vertices, call them $x$ and $y$.  Let $X$ denote the 8 vertices other than $y$ that are adjacent to $x$, and let $Y$ denote the 8 vertices other than $x$ adjacent to $y$.   Since $\lambda = 0$, the sets $X$ and $Y$ have empty intersection.  Let $Z$ consist of the ten other vertices not in $\{x\} \cup \{y\} \cup X \cup Y$.  Let $z$ be a vertex in $Z$.  Since $\mu = 4$, four edges from $z$ must enter $X$ and four edges must enter $Y$.  This leaves one edge to connect $z$ to another vertex in $Z$.  It follows that the induced subgraph on $Z$ is five disjoint copies of $P_2$, the path graph on two vertices (i.e., an edge between two vertices).  Adding in vertices $x$ and $y$, the induced subgraph is then six copies of $P_2$.  

Each of these copies of $P_2$ can be used to build a vector in $\ker \overline{L}$.  The matrix equation \ref{eqn1} shows us how: Working modulo 7, the equation reads: $\overline{L}(\overline{L} - I) = 4J.$  Thus $\overline{L}$ maps any column of $\overline{L} - I$ to $4\allone$, where $\allone$ is the vector of all-ones.  Thus, the difference of any two columns of the $\overline{L} - I$ will be in $\ker \overline{L}$.  To be concrete, supposed we built our Laplacian matrix by ordering the vertices as follows:  $x$, $y$, then the vertices in $Z$, then the vertices in $X$, then the vertices in $Y$.  Take the column of $\overline{L} - I$ that is indexed by $x$ and the column that is indexed by $y$ and subtract them. The result, still working modulo 7, is expressed in the first column of the following matrix (we discuss the remaining columns momentarily). 

\[
C = \begin{bmatrix} 
2  &0&0&0&0&0\\
-2 &0&0&0&0&0\\
0&2&0&0&0&0\\
0&-2&0&0&0&0\\
0&0&2&0&0&0\\
0&0&-2&0&0&0\\
0&0&0&2&0&0\\
0&0&0&-2&0&0\\
0&0&0&0&2&0\\
0&0&0&0&-2&0\\
0&0&0&0&0&2\\
0&0&0&0&0&-2\\
{\bf -1}_8&{\bf ?}_8 & {\bf ?}_8 & {\bf ?}_8& {\bf ?}_8& {\bf ?}_8\\
{\bf 1}_8&{\bf ?}_8 & {\bf ?}_8 & {\bf ?}_8& {\bf ?}_8& {\bf ?}_8\\
\end{bmatrix}
\]
Here ${\bf k}_8$ denotes $8$ repeated vertical entries of the number $k$, and ${\bf ?}_8$ denotes $8$ vertical entries with unknown value.


Suppose further that we ordered the vertices so that the next two vertices (which are in $Z$) are adjacent, and the two vertices after that (still in $Z$) are adjacent, etc.  Then as we just considered the difference between the first and second columns of $\overline{L} - I$, also consider the difference between the third and fourth, fifth and sixth, $\ldots$, eleventh and twelfth.  If we throw all of these six columns into a matrix, we obtain the matrix $C$ above.

Clearly these six columns are independent and so form a basis for $\ker \overline{L}$.  But don't forget that $\allone$ is also in $\ker \overline{L}$, and (as is not hard to check) is not an $F$-linear combination of these six vectors. Thus we have seven vectors in the kernel, which is a contradiction to our dimension count above.
\end{example}

In the example above, all that was really used was the $7$-rank of $L$ (which can be obtained from \cite{brouwer:prank}); we did not need the full information given by the critical group.  Perhaps a more sophisticated use of these strategies can employ the other information in the Smith normal form to eliminate further parameter sets.

\section{Acknowledgements}
We are grateful for support from the National Science Foundation (grant number NSF-DMS 1560151).

\section{Appendix}
We include in this appendix feasible parameter sets for strongly regular graphs with nonzero integer Laplacian eigenvalues $r_{L}$ and $s_{L}$, for graphs with less than 200 vertices, so that the reader may easily apply the results of the paper.  See Andries Brouwer's website \cite{brouwer:website} or the Handbook of Combinatorial Designs \cite{handbook} for more detailed information, including graph constructions and existence data.  Note that in those sources, $r$ and $s$ refer to eigenvalues of an adjacency matrix of such a graph.

\scalebox{.6}{	
											
\begin{centering}
\begin{tabular}{rrrr|rr|rrr||rrrrr|rr|rrr||rrrrr|rr|rr}
$v$	&	$k$	&	$\lambda$	&	$\mu$	&	$r_L$	&	$f$	&	$s_L$	&	$g$	&  &  &     $v$	&	$k$	&	$\lambda$	&	$\mu$	&	$r_L$	&	$f$	&	$s_L$	&	$g$	& & &     $v$	&	$k$	&	$\lambda$	&	$\mu$	&	$r_L$	&	$f$	&	$s_L$	&	$g$\\ \hline
9	&	4	&	1	&	2	&	3	&	4	&	6	&	4	&	&	&	64	&	49	&	36	&	42	&	48	&	49	&	56	&	14	&	&	&	96	&	45	&	24	&	18	&	36	&	20	&	48	&	75	\\
10	&	3	&	0	&	1	&	2	&	5	&	5	&	4	&	&	&	64	&	18	&	2	&	6	&	16	&	45	&	24	&	18	&	&	&	96	&	50	&	22	&	30	&	48	&	75	&	60	&	20	\\
10	&	6	&	3	&	4	&	5	&	4	&	8	&	5	&	&	&	64	&	45	&	32	&	30	&	40	&	18	&	48	&	45	&	&	&	99	&	14	&	1	&	2	&	11	&	54	&	18	&	44	\\
15	&	6	&	1	&	3	&	5	&	9	&	9	&	5	&	&	&	64	&	21	&	0	&	10	&	20	&	56	&	32	&	7	&	&	&	99	&	84	&	71	&	72	&	81	&	44	&	88	&	54	\\
15	&	8	&	4	&	4	&	6	&	5	&	10	&	9	&	&	&	64	&	42	&	30	&	22	&	32	&	7	&	44	&	56	&	&	&	99	&	42	&	21	&	15	&	33	&	21	&	45	&	77	\\
16	&	5	&	0	&	2	&	4	&	10	&	8	&	5	&	&	&	64	&	21	&	8	&	6	&	16	&	21	&	24	&	42	&	&	&	99	&	56	&	28	&	36	&	54	&	77	&	66	&	21	\\
16	&	10	&	6	&	6	&	8	&	5	&	12	&	10	&	&	&	64	&	42	&	26	&	30	&	40	&	42	&	48	&	21	&	&	&	99	&	48	&	22	&	24	&	44	&	54	&	54	&	44	\\
16	&	6	&	2	&	2	&	4	&	6	&	8	&	9	&	&	&	64	&	27	&	10	&	12	&	24	&	36	&	32	&	27	&	&	&	99	&	50	&	25	&	25	&	45	&	44	&	55	&	54	\\
16	&	9	&	4	&	6	&	8	&	9	&	12	&	6	&	&	&	64	&	36	&	20	&	20	&	32	&	27	&	40	&	36	&	&	&	100	&	18	&	8	&	2	&	10	&	18	&	20	&	81	\\
21	&	10	&	3	&	6	&	9	&	14	&	14	&	6	&	&	&	64	&	28	&	12	&	12	&	24	&	28	&	32	&	35	&	&	&	100	&	81	&	64	&	72	&	80	&	81	&	90	&	18	\\
21	&	10	&	5	&	4	&	7	&	6	&	12	&	14	&	&	&	64	&	35	&	18	&	20	&	32	&	35	&	40	&	28	&	&	&	100	&	22	&	0	&	6	&	20	&	77	&	30	&	22	\\
25	&	8	&	3	&	2	&	5	&	8	&	10	&	16	&	&	&	64	&	30	&	18	&	10	&	20	&	8	&	32	&	55	&	&	&	100	&	77	&	60	&	56	&	70	&	22	&	80	&	77	\\
25	&	16	&	9	&	12	&	15	&	16	&	20	&	8	&	&	&	64	&	33	&	12	&	22	&	32	&	55	&	44	&	8	&	&	&	100	&	27	&	10	&	6	&	20	&	27	&	30	&	72	\\
25	&	12	&	5	&	6	&	10	&	12	&	15	&	12	&	&	&	66	&	20	&	10	&	4	&	12	&	11	&	22	&	54	&	&	&	100	&	72	&	50	&	56	&	70	&	72	&	80	&	27	\\
26	&	10	&	3	&	4	&	8	&	13	&	13	&	12	&	&	&	66	&	45	&	28	&	36	&	44	&	54	&	54	&	11	&	&	&	100	&	33	&	8	&	12	&	30	&	66	&	40	&	33	\\
26	&	15	&	8	&	9	&	13	&	12	&	18	&	13	&	&	&	69	&	20	&	7	&	5	&	15	&	23	&	23	&	45	&	&	&	100	&	66	&	44	&	42	&	60	&	33	&	70	&	66	\\
27	&	10	&	1	&	5	&	9	&	20	&	15	&	6	&	&	&	69	&	48	&	32	&	36	&	46	&	45	&	54	&	23	&	&	&	100	&	33	&	14	&	9	&	25	&	24	&	36	&	75	\\
27	&	16	&	10	&	8	&	12	&	6	&	18	&	20	&	&	&	70	&	27	&	12	&	9	&	21	&	20	&	30	&	49	&	&	&	100	&	66	&	41	&	48	&	64	&	75	&	75	&	24	\\
28	&	9	&	0	&	4	&	8	&	21	&	14	&	6	&	&	&	70	&	42	&	23	&	28	&	40	&	49	&	49	&	20	&	&	&	100	&	33	&	18	&	7	&	20	&	11	&	35	&	88	\\
28	&	18	&	12	&	10	&	14	&	6	&	20	&	21	&	&	&	75	&	32	&	10	&	16	&	30	&	56	&	40	&	18	&	&	&	100	&	66	&	39	&	52	&	65	&	88	&	80	&	11	\\
28	&	12	&	6	&	4	&	8	&	7	&	14	&	20	&	&	&	75	&	42	&	25	&	21	&	35	&	18	&	45	&	56	&	&	&	100	&	36	&	14	&	12	&	30	&	36	&	40	&	63	\\
28	&	15	&	6	&	10	&	14	&	20	&	20	&	7	&	&	&	76	&	21	&	2	&	7	&	19	&	56	&	28	&	19	&	&	&	100	&	63	&	38	&	42	&	60	&	63	&	70	&	36	\\
35	&	16	&	6	&	8	&	14	&	20	&	20	&	14	&	&	&	76	&	54	&	39	&	36	&	48	&	19	&	57	&	56	&	&	&	100	&	44	&	18	&	20	&	40	&	55	&	50	&	44	\\
35	&	18	&	9	&	9	&	15	&	14	&	21	&	20	&	&	&	76	&	30	&	8	&	14	&	28	&	57	&	38	&	18	&	&	&	100	&	55	&	30	&	30	&	50	&	44	&	60	&	55	\\
36	&	10	&	4	&	2	&	6	&	10	&	12	&	25	&	&	&	76	&	45	&	28	&	24	&	38	&	18	&	48	&	57	&	&	&	100	&	45	&	20	&	20	&	40	&	45	&	50	&	54	\\
36	&	25	&	16	&	20	&	24	&	25	&	30	&	10	&	&	&	76	&	35	&	18	&	14	&	28	&	19	&	38	&	56	&	&	&	100	&	54	&	28	&	30	&	50	&	54	&	60	&	45	\\
36	&	14	&	4	&	6	&	12	&	21	&	18	&	14	&	&	&	76	&	40	&	18	&	24	&	38	&	56	&	48	&	19	&	&	&	105	&	26	&	13	&	4	&	15	&	14	&	28	&	90	\\
36	&	21	&	12	&	12	&	18	&	14	&	24	&	21	&	&	&	77	&	16	&	0	&	4	&	14	&	55	&	22	&	21	&	&	&	105	&	78	&	55	&	66	&	77	&	90	&	90	&	14	\\
36	&	14	&	7	&	4	&	9	&	8	&	16	&	27	&	&	&	77	&	60	&	47	&	45	&	55	&	21	&	63	&	55	&	&	&	105	&	32	&	4	&	12	&	30	&	84	&	42	&	20	\\
36	&	21	&	10	&	15	&	20	&	27	&	27	&	8	&	&	&	78	&	22	&	11	&	4	&	13	&	12	&	24	&	65	&	&	&	105	&	72	&	51	&	45	&	63	&	20	&	75	&	84	\\
36	&	15	&	6	&	6	&	12	&	15	&	18	&	20	&	&	&	78	&	55	&	36	&	45	&	54	&	65	&	65	&	12	&	&	&	105	&	40	&	15	&	15	&	35	&	48	&	45	&	56	\\
36	&	20	&	10	&	12	&	18	&	20	&	24	&	15	&	&	&	81	&	16	&	7	&	2	&	9	&	16	&	18	&	64	&	&	&	105	&	64	&	38	&	40	&	60	&	56	&	70	&	48	\\
40	&	12	&	2	&	4	&	10	&	24	&	16	&	15	&	&	&	81	&	64	&	49	&	56	&	63	&	64	&	72	&	16	&	&	&	105	&	52	&	21	&	30	&	50	&	84	&	63	&	20	\\
40	&	27	&	18	&	18	&	24	&	15	&	30	&	24	&	&	&	81	&	20	&	1	&	6	&	18	&	60	&	27	&	20	&	&	&	105	&	52	&	29	&	22	&	42	&	20	&	55	&	84	\\
45	&	12	&	3	&	3	&	9	&	20	&	15	&	24	&	&	&	81	&	60	&	45	&	42	&	54	&	20	&	63	&	60	&	&	&	111	&	30	&	5	&	9	&	27	&	74	&	37	&	36	\\
45	&	32	&	22	&	24	&	30	&	24	&	36	&	20	&	&	&	81	&	24	&	9	&	6	&	18	&	24	&	27	&	56	&	&	&	111	&	80	&	58	&	56	&	74	&	36	&	84	&	74	\\
45	&	16	&	8	&	4	&	10	&	9	&	18	&	35	&	&	&	81	&	56	&	37	&	42	&	54	&	56	&	63	&	24	&	&	&	111	&	44	&	19	&	16	&	37	&	36	&	48	&	74	\\
45	&	28	&	15	&	21	&	27	&	35	&	35	&	9	&	&	&	81	&	30	&	9	&	12	&	27	&	50	&	36	&	30	&	&	&	111	&	66	&	37	&	42	&	63	&	74	&	74	&	36	\\
49	&	12	&	5	&	2	&	7	&	12	&	14	&	36	&	&	&	81	&	50	&	31	&	30	&	45	&	30	&	54	&	50	&	&	&	112	&	30	&	2	&	10	&	28	&	90	&	40	&	21	\\
49	&	36	&	25	&	30	&	35	&	36	&	42	&	12	&	&	&	81	&	32	&	13	&	12	&	27	&	32	&	36	&	48	&	&	&	112	&	81	&	60	&	54	&	72	&	21	&	84	&	90	\\
49	&	16	&	3	&	6	&	14	&	32	&	21	&	16	&	&	&	81	&	48	&	27	&	30	&	45	&	48	&	54	&	32	&	&	&	112	&	36	&	10	&	12	&	32	&	63	&	42	&	48	\\
49	&	32	&	21	&	20	&	28	&	16	&	35	&	32	&	&	&	81	&	40	&	13	&	26	&	39	&	72	&	54	&	8	&	&	&	112	&	75	&	50	&	50	&	70	&	48	&	80	&	63	\\
49	&	18	&	7	&	6	&	14	&	18	&	21	&	30	&	&	&	81	&	40	&	25	&	14	&	27	&	8	&	42	&	72	&	&	&	115	&	18	&	1	&	3	&	15	&	69	&	23	&	45	\\
49	&	30	&	17	&	20	&	28	&	30	&	35	&	18	&	&	&	81	&	40	&	19	&	20	&	36	&	40	&	45	&	40	&	&	&	115	&	96	&	80	&	80	&	92	&	45	&	100	&	69	\\
49	&	24	&	11	&	12	&	21	&	24	&	28	&	24	&	&	&	82	&	36	&	15	&	16	&	32	&	41	&	41	&	40	&	&	&	117	&	36	&	15	&	9	&	27	&	26	&	39	&	90	\\
50	&	7	&	0	&	1	&	5	&	28	&	10	&	21	&	&	&	82	&	45	&	24	&	25	&	41	&	40	&	50	&	41	&	&	&	117	&	80	&	52	&	60	&	78	&	90	&	90	&	26	\\
50	&	42	&	35	&	36	&	40	&	21	&	45	&	28	&	&	&	85	&	14	&	3	&	2	&	10	&	34	&	17	&	50	&	&	&	119	&	54	&	21	&	27	&	51	&	84	&	63	&	34	\\
50	&	21	&	4	&	12	&	20	&	42	&	30	&	7	&	&	&	85	&	70	&	57	&	60	&	68	&	50	&	75	&	34	&	&	&	119	&	64	&	36	&	32	&	56	&	34	&	68	&	84	\\
50	&	28	&	18	&	12	&	20	&	7	&	30	&	42	&	&	&	85	&	20	&	3	&	5	&	17	&	50	&	25	&	34	&	&	&	120	&	28	&	14	&	4	&	16	&	15	&	30	&	104	\\
50	&	21	&	8	&	9	&	18	&	25	&	25	&	24	&	&	&	85	&	64	&	48	&	48	&	60	&	34	&	68	&	50	&	&	&	120	&	91	&	66	&	78	&	90	&	104	&	104	&	15	\\
50	&	28	&	15	&	16	&	25	&	24	&	32	&	25	&	&	&	85	&	30	&	11	&	10	&	25	&	34	&	34	&	50	&	&	&	120	&	34	&	8	&	10	&	30	&	68	&	40	&	51	\\
55	&	18	&	9	&	4	&	11	&	10	&	20	&	44	&	&	&	85	&	54	&	33	&	36	&	51	&	50	&	60	&	34	&	&	&	120	&	85	&	60	&	60	&	80	&	51	&	90	&	68	\\
55	&	36	&	21	&	28	&	35	&	44	&	44	&	10	&	&	&	88	&	27	&	6	&	9	&	24	&	55	&	33	&	32	&	&	&	120	&	35	&	10	&	10	&	30	&	56	&	40	&	63	\\
56	&	10	&	0	&	2	&	8	&	35	&	14	&	20	&	&	&	88	&	60	&	41	&	40	&	55	&	32	&	64	&	55	&	&	&	120	&	84	&	58	&	60	&	80	&	63	&	90	&	56	\\
56	&	45	&	36	&	36	&	42	&	20	&	48	&	35	&	&	&	91	&	24	&	12	&	4	&	14	&	13	&	26	&	77	&	&	&	120	&	42	&	8	&	18	&	40	&	99	&	54	&	20	\\
56	&	22	&	3	&	12	&	21	&	48	&	32	&	7	&	&	&	91	&	66	&	45	&	55	&	65	&	77	&	77	&	13	&	&	&	120	&	77	&	52	&	44	&	66	&	20	&	80	&	99	\\
56	&	33	&	22	&	15	&	24	&	7	&	35	&	48	&	&	&	95	&	40	&	12	&	20	&	38	&	75	&	50	&	19	&	&	&	120	&	51	&	18	&	24	&	48	&	85	&	60	&	34	\\
57	&	14	&	1	&	4	&	12	&	38	&	19	&	18	&	&	&	95	&	54	&	33	&	27	&	45	&	19	&	57	&	75	&	&	&	120	&	68	&	40	&	36	&	60	&	34	&	72	&	85	\\
57	&	42	&	31	&	30	&	38	&	18	&	45	&	38	&	&	&	96	&	19	&	2	&	4	&	16	&	57	&	24	&	38	&	&	&	120	&	56	&	28	&	24	&	48	&	35	&	60	&	84	\\
57	&	24	&	11	&	9	&	19	&	18	&	27	&	38	&	&	&	96	&	76	&	60	&	60	&	72	&	38	&	80	&	57	&	&	&	120	&	63	&	30	&	36	&	60	&	84	&	72	&	35	\\
57	&	32	&	16	&	20	&	30	&	38	&	38	&	18	&	&	&	96	&	20	&	4	&	4	&	16	&	45	&	24	&	50	&	&	&	121	&	20	&	9	&	2	&	11	&	20	&	22	&	100	\\
63	&	22	&	1	&	11	&	21	&	55	&	33	&	7	&	&	&	96	&	75	&	58	&	60	&	72	&	50	&	80	&	45	&	&	&	121	&	100	&	81	&	90	&	99	&	100	&	110	&	20	\\
63	&	40	&	28	&	20	&	30	&	7	&	42	&	55	&	&	&	96	&	35	&	10	&	14	&	32	&	63	&	42	&	32	&	&	&	121	&	30	&	11	&	6	&	22	&	30	&	33	&	90	\\
63	&	30	&	13	&	15	&	27	&	35	&	35	&	27	&	&	&	96	&	60	&	38	&	36	&	54	&	32	&	64	&	63	&	&	&	121	&	90	&	65	&	72	&	88	&	90	&	99	&	30	\\
63	&	32	&	16	&	16	&	28	&	27	&	36	&	35	&	&	&	96	&	38	&	10	&	18	&	36	&	76	&	48	&	19	&	&	&	121	&	36	&	7	&	12	&	33	&	84	&	44	&	36	\\
64	&	14	&	6	&	2	&	8	&	14	&	16	&	49	&	&	&	96	&	57	&	36	&	30	&	48	&	19	&	60	&	76	&	&	&	121	&	84	&	59	&	56	&	77	&	36	&	88	&	84	
\end{tabular}
\end{centering}

}

\scalebox{.55}{	
											
\begin{centering}
\begin{tabular}{rrrr|rr|rrr||rrrrr|rr|rrr||rrrrr|rr|rr}
$v$	&	$k$	&	$\lambda$	&	$\mu$	&	$r_L$	&	$f$	&	$s_L$	&	$g$	&  &  &     $v$	&	$k$	&	$\lambda$	&	$\mu$	&	$r_L$	&	$f$	&	$s_L$	&	$g$	& & &     $v$	&	$k$	&	$\lambda$	&	$\mu$	&	$r_L$	&	$f$	&	$s_L$	&	$g$\\ \hline
121	&	40	&	15	&	12	&	33	&	40	&	44	&	80	&	&	&	148	&	84	&	50	&	44	&	74	&	36	&	88	&	111	&	&	&	176	&	45	&	18	&	9	&	33	&	32	&	48	&	143	\\
121	&	80	&	51	&	56	&	77	&	80	&	88	&	40	&	&	&	148	&	70	&	36	&	30	&	60	&	37	&	74	&	110	&	&	&	176	&	130	&	93	&	104	&	128	&	143	&	143	&	32	\\
121	&	48	&	17	&	20	&	44	&	72	&	55	&	48	&	&	&	148	&	77	&	36	&	44	&	74	&	110	&	88	&	37	&	&	&	176	&	49	&	12	&	14	&	44	&	98	&	56	&	77	\\
121	&	72	&	43	&	42	&	66	&	48	&	77	&	72	&	&	&	153	&	32	&	16	&	4	&	18	&	17	&	34	&	135	&	&	&	176	&	126	&	90	&	90	&	120	&	77	&	132	&	98	\\
121	&	50	&	21	&	20	&	44	&	50	&	55	&	70	&	&	&	153	&	120	&	91	&	105	&	119	&	135	&	135	&	17	&	&	&	176	&	70	&	18	&	34	&	68	&	154	&	88	&	21	\\
121	&	70	&	39	&	42	&	66	&	70	&	77	&	50	&	&	&	153	&	56	&	19	&	21	&	51	&	84	&	63	&	68	&	&	&	176	&	105	&	68	&	54	&	88	&	21	&	108	&	154	\\
121	&	56	&	15	&	35	&	55	&	112	&	77	&	8	&	&	&	153	&	96	&	60	&	60	&	90	&	68	&	102	&	84	&	&	&	176	&	70	&	24	&	30	&	66	&	120	&	80	&	55	\\
121	&	64	&	42	&	24	&	44	&	8	&	66	&	112	&	&	&	154	&	48	&	12	&	16	&	44	&	98	&	56	&	55	&	&	&	176	&	105	&	64	&	60	&	96	&	55	&	110	&	120	\\
121	&	60	&	29	&	30	&	55	&	60	&	66	&	60	&	&	&	154	&	105	&	72	&	70	&	98	&	55	&	110	&	98	&	&	&	176	&	70	&	42	&	18	&	44	&	10	&	72	&	165	\\
122	&	55	&	24	&	25	&	50	&	61	&	61	&	60	&	&	&	154	&	51	&	8	&	21	&	49	&	132	&	66	&	21	&	&	&	176	&	105	&	52	&	78	&	104	&	165	&	132	&	10	\\
122	&	66	&	35	&	36	&	61	&	60	&	72	&	61	&	&	&	154	&	102	&	71	&	60	&	88	&	21	&	105	&	132	&	&	&	176	&	85	&	48	&	34	&	68	&	22	&	88	&	153	\\
125	&	28	&	3	&	7	&	25	&	84	&	35	&	40	&	&	&	154	&	72	&	26	&	40	&	70	&	132	&	88	&	21	&	&	&	176	&	90	&	38	&	54	&	88	&	153	&	108	&	22	\\
125	&	96	&	74	&	72	&	90	&	40	&	100	&	84	&	&	&	154	&	81	&	48	&	36	&	66	&	21	&	84	&	132	&	&	&	183	&	52	&	11	&	16	&	48	&	122	&	61	&	60	\\
125	&	48	&	28	&	12	&	30	&	10	&	50	&	114	&	&	&	155	&	42	&	17	&	9	&	31	&	30	&	45	&	124	&	&	&	183	&	130	&	93	&	90	&	122	&	60	&	135	&	122	\\
125	&	76	&	39	&	57	&	75	&	114	&	95	&	10	&	&	&	155	&	112	&	78	&	88	&	110	&	124	&	124	&	30	&	&	&	183	&	70	&	29	&	25	&	61	&	60	&	75	&	122	\\
125	&	52	&	15	&	26	&	50	&	104	&	65	&	20	&	&	&	156	&	30	&	4	&	6	&	26	&	90	&	36	&	65	&	&	&	183	&	112	&	66	&	72	&	108	&	122	&	122	&	60	\\
125	&	72	&	45	&	36	&	60	&	20	&	75	&	104	&	&	&	156	&	125	&	100	&	100	&	120	&	65	&	130	&	90	&	&	&	184	&	48	&	2	&	16	&	46	&	160	&	64	&	23	\\
126	&	25	&	8	&	4	&	18	&	35	&	28	&	90	&	&	&	160	&	54	&	18	&	18	&	48	&	75	&	60	&	84	&	&	&	184	&	135	&	102	&	90	&	120	&	23	&	138	&	160	\\
126	&	100	&	78	&	84	&	98	&	90	&	108	&	35	&	&	&	160	&	105	&	68	&	70	&	100	&	84	&	112	&	75	&	&	&	189	&	48	&	12	&	12	&	42	&	90	&	54	&	98	\\
126	&	45	&	12	&	18	&	42	&	90	&	54	&	35	&	&	&	162	&	21	&	0	&	3	&	18	&	105	&	27	&	56	&	&	&	189	&	140	&	103	&	105	&	135	&	98	&	147	&	90	\\
126	&	80	&	52	&	48	&	72	&	35	&	84	&	90	&	&	&	162	&	140	&	121	&	120	&	135	&	56	&	144	&	105	&	&	&	189	&	60	&	27	&	15	&	45	&	28	&	63	&	160	\\
126	&	50	&	13	&	24	&	48	&	105	&	63	&	20	&	&	&	162	&	23	&	4	&	3	&	18	&	69	&	27	&	92	&	&	&	189	&	128	&	82	&	96	&	126	&	160	&	144	&	28	\\
126	&	75	&	48	&	39	&	63	&	20	&	78	&	105	&	&	&	162	&	138	&	117	&	120	&	135	&	92	&	144	&	69	&	&	&	189	&	88	&	37	&	44	&	84	&	132	&	99	&	56	\\
126	&	60	&	33	&	24	&	48	&	21	&	63	&	104	&	&	&	162	&	49	&	16	&	14	&	42	&	63	&	54	&	98	&	&	&	189	&	100	&	55	&	50	&	90	&	56	&	105	&	132	\\
126	&	65	&	28	&	39	&	63	&	104	&	78	&	21	&	&	&	162	&	112	&	76	&	80	&	108	&	98	&	120	&	63	&	&	&	190	&	36	&	18	&	4	&	20	&	19	&	38	&	170	\\
130	&	48	&	20	&	16	&	40	&	39	&	52	&	90	&	&	&	162	&	56	&	10	&	24	&	54	&	140	&	72	&	21	&	&	&	190	&	153	&	120	&	136	&	152	&	170	&	170	&	19	\\
130	&	81	&	48	&	54	&	78	&	90	&	90	&	39	&	&	&	162	&	105	&	72	&	60	&	90	&	21	&	108	&	140	&	&	&	190	&	45	&	12	&	10	&	38	&	75	&	50	&	114	\\
133	&	24	&	5	&	4	&	19	&	56	&	28	&	76	&	&	&	162	&	69	&	36	&	24	&	54	&	23	&	72	&	138	&	&	&	190	&	144	&	108	&	112	&	140	&	114	&	152	&	75	\\
133	&	108	&	87	&	90	&	105	&	76	&	114	&	56	&	&	&	162	&	92	&	46	&	60	&	90	&	138	&	108	&	23	&	&	&	190	&	84	&	33	&	40	&	80	&	133	&	95	&	56	\\
133	&	32	&	6	&	8	&	28	&	76	&	38	&	56	&	&	&	165	&	36	&	3	&	9	&	33	&	120	&	45	&	44	&	&	&	190	&	105	&	60	&	55	&	95	&	56	&	110	&	133	\\
133	&	100	&	75	&	75	&	95	&	56	&	105	&	76	&	&	&	165	&	128	&	100	&	96	&	120	&	44	&	132	&	120	&	&	&	190	&	84	&	38	&	36	&	76	&	75	&	90	&	114	\\
133	&	44	&	15	&	14	&	38	&	56	&	49	&	76	&	&	&	169	&	24	&	11	&	2	&	13	&	24	&	26	&	144	&	&	&	190	&	105	&	56	&	60	&	100	&	114	&	114	&	75	\\
133	&	88	&	57	&	60	&	84	&	76	&	95	&	56	&	&	&	169	&	144	&	121	&	132	&	143	&	144	&	156	&	24	&	&	&	190	&	90	&	45	&	40	&	80	&	57	&	95	&	132	\\
135	&	64	&	28	&	32	&	60	&	84	&	72	&	50	&	&	&	169	&	36	&	13	&	6	&	26	&	36	&	39	&	132	&	&	&	190	&	99	&	48	&	55	&	95	&	132	&	110	&	57	\\
135	&	70	&	37	&	35	&	63	&	50	&	75	&	84	&	&	&	169	&	132	&	101	&	110	&	130	&	132	&	143	&	36	&	&	&	195	&	96	&	46	&	48	&	90	&	104	&	104	&	90	\\
136	&	30	&	8	&	6	&	24	&	51	&	34	&	84	&	&	&	169	&	42	&	5	&	12	&	39	&	126	&	52	&	42	&	&	&	195	&	98	&	49	&	49	&	91	&	90	&	105	&	104	\\
136	&	105	&	80	&	84	&	102	&	84	&	112	&	51	&	&	&	169	&	126	&	95	&	90	&	117	&	42	&	130	&	126	&	&	&	196	&	26	&	12	&	2	&	14	&	26	&	28	&	169	\\
136	&	30	&	15	&	4	&	17	&	16	&	32	&	119	&	&	&	169	&	48	&	17	&	12	&	39	&	48	&	52	&	120	&	&	&	196	&	169	&	144	&	156	&	168	&	169	&	182	&	26	\\
136	&	105	&	78	&	91	&	104	&	119	&	119	&	16	&	&	&	169	&	120	&	83	&	90	&	117	&	120	&	130	&	48	&	&	&	196	&	39	&	2	&	9	&	36	&	147	&	49	&	48	\\
136	&	60	&	24	&	28	&	56	&	85	&	68	&	50	&	&	&	169	&	56	&	15	&	20	&	52	&	112	&	65	&	56	&	&	&	196	&	156	&	125	&	120	&	147	&	48	&	160	&	147	\\
136	&	75	&	42	&	40	&	68	&	50	&	80	&	85	&	&	&	169	&	112	&	75	&	72	&	104	&	56	&	117	&	112	&	&	&	196	&	39	&	14	&	6	&	28	&	39	&	42	&	156	\\
136	&	63	&	30	&	28	&	56	&	51	&	68	&	84	&	&	&	169	&	60	&	23	&	20	&	52	&	60	&	65	&	108	&	&	&	196	&	156	&	122	&	132	&	154	&	156	&	168	&	39	\\
136	&	72	&	36	&	40	&	68	&	84	&	80	&	51	&	&	&	169	&	108	&	67	&	72	&	104	&	108	&	117	&	60	&	&	&	196	&	45	&	4	&	12	&	42	&	150	&	56	&	45	\\
143	&	70	&	33	&	35	&	65	&	77	&	77	&	65	&	&	&	169	&	70	&	27	&	30	&	65	&	98	&	78	&	70	&	&	&	196	&	150	&	116	&	110	&	140	&	45	&	154	&	150	\\
143	&	72	&	36	&	36	&	66	&	65	&	78	&	77	&	&	&	169	&	98	&	57	&	56	&	91	&	70	&	104	&	98	&	&	&	196	&	52	&	18	&	12	&	42	&	52	&	56	&	143	\\
144	&	22	&	10	&	2	&	12	&	22	&	24	&	121	&	&	&	169	&	72	&	31	&	30	&	65	&	72	&	78	&	96	&	&	&	196	&	143	&	102	&	110	&	140	&	143	&	154	&	52	\\
144	&	121	&	100	&	110	&	120	&	121	&	132	&	22	&	&	&	169	&	96	&	53	&	56	&	91	&	96	&	104	&	72	&	&	&	196	&	60	&	14	&	20	&	56	&	135	&	70	&	60	\\
144	&	33	&	12	&	6	&	24	&	33	&	36	&	110	&	&	&	169	&	84	&	41	&	42	&	78	&	84	&	91	&	84	&	&	&	196	&	135	&	94	&	90	&	126	&	60	&	140	&	135	\\
144	&	110	&	82	&	90	&	108	&	110	&	120	&	33	&	&	&	170	&	78	&	35	&	36	&	72	&	85	&	85	&	84	&	&	&	196	&	60	&	23	&	16	&	49	&	48	&	64	&	147	\\
144	&	39	&	6	&	12	&	36	&	104	&	48	&	39	&	&	&	170	&	91	&	48	&	49	&	85	&	84	&	98	&	85	&	&	&	196	&	135	&	90	&	99	&	132	&	147	&	147	&	48	\\
144	&	104	&	76	&	72	&	96	&	39	&	108	&	104	&	&	&	171	&	34	&	17	&	4	&	19	&	18	&	36	&	152	&	&	&	196	&	65	&	24	&	20	&	56	&	65	&	70	&	130	\\
144	&	44	&	16	&	12	&	36	&	44	&	48	&	99	&	&	&	171	&	136	&	105	&	120	&	135	&	152	&	152	&	18	&	&	&	196	&	130	&	84	&	90	&	126	&	130	&	140	&	65	\\
144	&	99	&	66	&	72	&	96	&	99	&	108	&	44	&	&	&	171	&	50	&	13	&	15	&	45	&	95	&	57	&	75	&	&	&	196	&	75	&	26	&	30	&	70	&	120	&	84	&	75	\\
144	&	52	&	16	&	20	&	48	&	91	&	60	&	52	&	&	&	171	&	120	&	84	&	84	&	114	&	75	&	126	&	95	&	&	&	196	&	120	&	74	&	72	&	112	&	75	&	126	&	120	\\
144	&	91	&	58	&	56	&	84	&	52	&	96	&	91	&	&	&	171	&	60	&	15	&	24	&	57	&	132	&	72	&	38	&	&	&	196	&	78	&	32	&	30	&	70	&	78	&	84	&	117	\\
144	&	55	&	22	&	20	&	48	&	55	&	60	&	88	&	&	&	171	&	110	&	73	&	66	&	99	&	38	&	114	&	132	&	&	&	196	&	117	&	68	&	72	&	112	&	117	&	126	&	78	\\
144	&	88	&	52	&	56	&	84	&	88	&	96	&	55	&	&	&	175	&	30	&	5	&	5	&	25	&	84	&	35	&	90	&	&	&	196	&	81	&	42	&	27	&	63	&	24	&	84	&	171	\\
144	&	65	&	16	&	40	&	64	&	135	&	90	&	8	&	&	&	175	&	144	&	118	&	120	&	140	&	90	&	150	&	84	&	&	&	196	&	114	&	59	&	76	&	112	&	171	&	133	&	24	\\
144	&	78	&	52	&	30	&	54	&	8	&	80	&	135	&	&	&	175	&	66	&	29	&	22	&	55	&	42	&	70	&	132	&	&	&	196	&	85	&	18	&	51	&	84	&	187	&	119	&	8	\\
144	&	65	&	28	&	30	&	60	&	78	&	72	&	65	&	&	&	175	&	108	&	63	&	72	&	105	&	132	&	120	&	42	&	&	&	196	&	110	&	75	&	44	&	77	&	8	&	112	&	187	\\
144	&	78	&	42	&	42	&	72	&	65	&	84	&	78	&	&	&	175	&	72	&	20	&	36	&	70	&	153	&	90	&	21	&	&	&	196	&	90	&	40	&	42	&	84	&	105	&	98	&	90	\\
144	&	66	&	30	&	30	&	60	&	66	&	72	&	77	&	&	&	175	&	102	&	65	&	51	&	85	&	21	&	105	&	153	&	&	&	196	&	105	&	56	&	56	&	98	&	90	&	112	&	105	\\
144	&	77	&	40	&	42	&	72	&	77	&	84	&	66	&	&	&	176	&	25	&	0	&	4	&	22	&	120	&	32	&	55	&	&	&	196	&	91	&	42	&	42	&	84	&	91	&	98	&	104	\\
147	&	66	&	25	&	33	&	63	&	110	&	77	&	36	&	&	&	176	&	150	&	128	&	126	&	144	&	55	&	154	&	120	&	&	&	196	&	104	&	54	&	56	&	98	&	104	&	112	&	91	\\
147	&	80	&	46	&	40	&	70	&	36	&	84	&	110	&	&	&	176	&	40	&	12	&	8	&	32	&	55	&	44	&	120	&	&	&		&		&		&		&		&		&		&		\\
148	&	63	&	22	&	30	&	60	&	111	&	74	&	36	&	&	&	176	&	135	&	102	&	108	&	132	&	120	&	144	&	55	&	&	&		&		&		&		&		&		&		&			
\end{tabular}
\end{centering}

}


\bibliographystyle{amsplain}

\bibliography{srgbib}

\end{document}